\date{\today}

\documentclass[10pt]{article}

\usepackage[T1]{fontenc}

\usepackage{amsmath}
\usepackage{amsthm}
\usepackage{amssymb}
\usepackage{url}
\usepackage{latexsym}
\usepackage{enumitem}

\usepackage{tikz}
\usepackage[square, sort, comma, numbers]{natbib}

\usepackage[small, it, singlelinecheck=false]{caption}

\usepackage[small, pagestyles]{titlesec}

\usepackage{color}
\definecolor{lightgray}{rgb}{0.8, 0.8, 0.8}
\definecolor{darkgray}{rgb}{0.65, 0.65, 0.65}

\usepackage[bookmarks]{hyperref}
\hypersetup{
        colorlinks=true,
        linkcolor=black,
        anchorcolor=black,
        citecolor=black,
        urlcolor=blue,
        pdfpagemode=UseThumbs,
        pdftitle={Letter graphs and modular decomposition},
        pdfsubject={Combinatorics},
        pdfauthor={Ferguson and Vatter},
}

\newcounter{todocounter}

\theoremstyle{plain}
\newtheorem{theorem}{Theorem}[section]

\newtheorem{proposition}[theorem]{Proposition}

\newtheorem{corollary}[theorem]{Corollary}

\theoremstyle{definition}

\makeatletter
\newfont{\footsc}{cmcsc10 at 8truept}
\newfont{\footbf}{cmbx10 at 8truept}
\newfont{\footrm}{cmr10 at 10truept}
\renewenvironment{abstract}%
                {
                  \begin{list}{}%
                     {\setlength{\rightmargin}{1in}%
                      \setlength{\leftmargin}{1in}}%
                   \item[]\ignorespaces\begin{small}}%
                 {\end{small}\unskip\end{list}}

\newcommand{\C}{\mathcal{C}}

\newcommand{\st}{\::\:}
\makeatletter

\newcommand\mybullet{\raisebox{-5pt}{\large \ensuremath{\bullet}}}
\newcommand\mycirc{\raisebox{-5pt}{\large \ensuremath{\circ}}}

\def\absdot{\@ifnextchar[{\@absdotlabel}{\@absdotnolabel}}
	\def\@absdotlabel[#1]#2{%
		\node at #2 {\normalsize \mybullet};
		\node at #2 [below=2pt] {\ensuremath{#1}};
	}
	\def\@absdotnolabel#1{%
		\node at #1 {\normalsize \mybullet};
	}
\def\absdothollow{\@ifnextchar[{\@absdothollowlabel}{\@absdothollownolabel}}
	\def\@absdothollowlabel[#1]#2{%
		\node at #2 {\normalsize \textcolor{white}{\mybullet}};
		\node at #2 {\normalsize \mycirc};
		\node at #2 [below=2pt] {\ensuremath{#1}};
	}
	\def\@absdothollownolabel#1{%
		\node at #1 {\normalsize \textcolor{white}{\mybullet}};
		\node at #1 {\normalsize \mycirc};
	}

\newcommand{\plotpartialperm}[1]{
	\foreach \i/\j in {#1} {
		\absdot{(\i,\j)};
	};
}

\makeatother

\makeatletter
\let\start@align@nopar\start@align
\let\start@gather@nopar\start@gather
\let\start@multline@nopar\start@multline
\long\def\start@align{\par\start@align@nopar}
\long\def\start@gather{\par\start@gather@nopar}
\long\def\start@multline{\par\start@multline@nopar}
\makeatother

\newpagestyle{main}[\small]{
        \headrule
        \sethead[\usepage][][]
        {\sc Letter Graphs and Modular Decomposition}{}{\usepage}}

\title{\sc Letter Graphs and Modular Decomposition}
\author{% And here it begins
	Robert Ferguson and Vincent Vatter%
	\footnote{Vatter's research was partially supported by the Simons Foundation via award number 636113.}\\[-0.25ex]\\[-0.25ex]
	\small Department of Mathematics\\[-0.5ex]
	\small University of Florida\\[-0.5ex]
	\small Gainesville, Florida USA
}

\titleformat{\section}
        {\large\sc}
        {\thesection.}{1em}{}

\begin{document}
\maketitle

% Symbol footnotes:
%\renewcommand*{\thefootnote}{\fnsymbol{footnote}}
%\addtocounter{footnote}{2}

\pagestyle{main}

\vspace{-0.3in}

\begin{abstract}
%We generalize a theorem of Alecu, Lozin, and de Werra, proving that if the prime graphs in a graph class have bounded lettericity, then the entire class has bounded lettericity if and only if it does not contain arbitrary large matchings, co-matchings, or a family of graphs that we call stacked paths.
We prove that if the prime graphs in a graph class have bounded lettericity, then the entire class has bounded lettericity if and only if it does not contain arbitrary large matchings, co-matchings, or a family of graphs that we call stacked paths.
\end{abstract}

\section{Introduction}
	
Our graphs are finite, simple, and undirected, and we denote the vertex set of the graph $G$ by $V(G)$. If $u$ and $v$ are adjacent vertices of $G$, then we write $u \sim v$; we write $u \nsim v$ otherwise. We extend this notation to disjoint subsets $U_1,U_2\subseteq V(G)$, writing $U_1\sim U_2$ if $u\sim v$ for every $u\in U_1$ and $v\in U_2$, and writing $U_1\nsim U_2$ if $u\nsim v$ for every $u\in U_1$ and $v\in U_2$. The neighborhood of a vertex $v \in V(G)$, denoted by $N(v)$, is the set of all vertices to which $v$ is adjacent.
	
We write $P_n$ for the path on $n$ vertices, $C_n$ for the cycle on $n$ vertices, and $K_n$ for the complete graph, or clique, on $n$ vertices. The complement of a graph $G$ is denoted by $\overline{G}$, and we refer to the complement of a clique as a \emph{co-clique}. Given graphs $G$ and $H$ on disjoint vertex sets, we denote their disjoint union by $G\uplus H$. Given a graph $G$ and a natural number $r$, we denote by $rG$ the disjoint union of $r$ copies of $G$ (where the copies are chosen to have disjoint vertex sets). For a positive integer $r$, we call the graph $rK_2$ a \emph{matching} and its complement $\overline{rK_2}$ a \emph{co-matching}.
	
A \emph{hereditary property} or (throughout this paper) \emph{class} of graphs is a set of finite graphs that is closed under isomorphism and also closed downward under the induced subgraph ordering. We are interested here in graph classes with bounded lettericity (defined in the next section). These classes are known to have many desirable propoerties. In particular, in his introduction of lettericity, Petkov\v{s}ek~\cite{petkovsek:letter-graphs-a:} proved that such classes are well-quasi-ordered by the induced subgraph order (an easy consequence of Higman's lemma), and Atminas and Lozin~\cite{atminas:labelled-induce:} have shown that classes of bounded lettericity are in fact \emph{labeled} well-quasi-ordered under the induced subgraph order. In addition to these order-theoretic considerations, there is a strong connection (first conjectured in~\cite{alecu:letter-graphs-a:abstract,alecu:letter-graphs-a:} and then established in~\cite{alecu:letter-graphs-a:iff}) between classes of bounded lettericity and the geometric grid classes of \cite{albert:geometric-grid-:,albert:inflations-of-g:} employed in the study of permutation patterns.

Our main result, Theorem~\ref{thm-prime-graphs-letters}, generalizes the following result to classes of graphs in which the prime graphs (those that cannot be decomposed via the modular decomposition) themselves have bounded lettericity. (For example, the class of cographs, defined in Section~\ref{sec-mod-decomp}, contains only three prime graphs, each of lettericity $1$).
		
\begin{theorem}[Alecu, Lozin, and De Werra~{\cite[Theorem~5]{alecu:the-micro-world:}}]
\label{thm-cographs-letters}
Let $\C$ be a class of cographs. If $\C$ contains all matchings or all co-matchings, then the lettericity of $\C$ is infinite. Otherwise, the lettericity of $\C$ is finite.
\end{theorem}

In Sections~\ref{sec-lettericity} and \ref{sec-mod-decomp}, respectively, we define lettericity and the modular decomposition. Section~\ref{sec-stacked-paths} defines the family of graphs we call stacked paths, and proves that they have unbounded lettericity. Finally, we state and prove our main result in Section~\ref{sec-main-theorem}.
	
\section{Lettericity}
\label{sec-lettericity}
	
%The concept of lettericity was introduced by Petkov\v{s}ek in~\cite{petkovsek:letter-graphs-a:}.
Let $\Sigma$ be a finite alphabet, and let ${D\subseteq\Sigma^2}$ be a set of ordered pairs that we call a \emph{decoder}. For any word ${w = w(1)w(2)\cdots w(n)}$ with each letter $w(i) \in \Sigma$, the \emph{letter graph of $w$ with respect to $D$} is the graph $\Gamma_D(w)$ with ${V(\Gamma_D(w)) = \{1,2,\dots,n\}}$ and in which for $i < j$, the vertices $i$ and $j$ are adjacent in $\Gamma_D(w)$ if and only if ${(w(i),w(j)) \in D}$.

If $\Sigma$ is an alphabet of cardinality $k$, then we say that $\Gamma_D(w)$ is a $k$-letter graph. For any graph $G$, the minimum $k$ such that a $G$ is a $k$-letter graph is the \emph{lettericity} of $G$, denoted by $\ell(G)$.  Every finite graph is the letter graph of some word over some alphabet, and in particular the lettericity of a graph $G$ is at most $|V(G)|$. Also note that $\ell(\overline{G})=\ell(G)$, as we may simply replace the decoder $D$ by its complement $\Sigma^2\setminus D$. 

Given a letter graph $\Gamma_D(w)$ and some letter $a \in \Sigma$, we then say that $a$ \emph{encodes} the set of vertices $\{i \in V(\Gamma_D(w)) \st w(i) = a\}$. Note that this set of vertices forms a clique if $(a,a) \in D$, and a co-clique otherwise. Given a graph $G$ such that $G = \Gamma_D(w)$, we say that $(D,w)$ is a \emph{lettering} of $G$, and in particular a \emph{$k$-lettering} if $w$ uses an alphabet of cardinality $k$.

The quintessential example of a class of graphs with bounded lettericity is the class of \emph{threshold graphs}. The threshold graphs have several definitions, but for our purposes the most useful is that they are the graphs that can be constructed by repeatedly adding new dominating vertices (adjacent to all of the vertices previously added) and isolated vertices (adjacent to none of the vertices previously added). We denote these two cases by $G * K_1$ and $G\uplus K_1$, respectively. Thus
\begin{enumerate}[topsep=-4pt, itemsep=-2pt]
	\item the empty graph $K_0$ is a threshold graph, and
	\item if $G$ is a threshold graph, then $G * K_1$ and $G \uplus K_1$ are threshold graphs.
\end{enumerate}
\vspace{6pt}

Equivalently, the threshold graphs are precisely the letter graphs on the alphabet $\Sigma=\{\mathsf{i},\mathsf{d}\}$ with the decoder $D=\{(\mathsf{i},\mathsf{d}),(\mathsf{d},\mathsf{d})\}$. To see this, simply encode vertices---in their order of addition to the graph---by $\mathsf{i}$ if they are added as isolated vertices or by $\mathsf{d}$ if they are added as dominating vertices.

We conclude this section with a result used later. Given two distinct vertices of a graph, a third vertex that is adjacent to one but not the other is said to \emph{distinguish} them.

\begin{proposition}
\label{prop-lettericity-dist}
	If a letter graph $\Gamma_D(w)$ has a pair of vertices $i < k$ with $w(i) = w(k)$, and this pair is distinguished by a third vertex $j$, then $i < j < k$.
\end{proposition}
\begin{proof}
	If it were the case that $i < k < j$ (resp., $j < i < k$), then the vertex $j$ of $\Gamma_D(w)$ would be adjacent to either both of the vertices $i$ and $k$ or neither of them, depending on whether ${(w(i),w(j)) \in D}$ (resp., ${(w(j),w(i)) \in D}$).
\end{proof}

\section{Modular Decomposition}
\label{sec-mod-decomp}

Here we briefly review modular decomposition, a concept that has been introduced numerous times in different contexts, but was brought to prominence in the graph context by the work of Gallai~\cite{gallai:transitiv-orien:,gallai:a-translation-o:}. Given a graph $G$, a \emph{module} of $G$ is a set $M \subseteq V(G)$ such that, for all $u,v \in M$, $N(u)\backslash M = N(v) \backslash M$. Note that every singleton is a module, as is the empty set and $V(G)$; we say that a module is \emph{proper} if it is not one of these. We then say that a graph is \emph{prime} if it contains no proper modules.
%We are concerned less with the modules of a graph and more with the prime graphs contained in a graph class.

We make use of two results about modular decomposition. The first is essentially the modular decomposition itself. Given a graph $G$, we say that $G$ is an inflation of the graph $H$ if $G$ can be obtained by replacing every vertex $v \in V(H)$ with its own graph $G_v$, with $V(G_u) \sim V(G_v)$ in $G$ if and only if $u \sim v$ in $H$. We write this as $G = H[G_v \st v \in V(H)]$.

\begin{theorem}
\label{thm-mod-decomp}
	For every graph $G$ on two or more vertices, there exists a unique prime graph $H$ on two or more vertices and a collection of nonempty graphs ${\{G_v\st v\in V(H)\}}$ such that $G$ is isomorphic to ${H[G_v \st v \in V(H)]}$.
\end{theorem}
	
We define the \emph{bull} to be the (prime) graph on five vertices consisting of a $P_4$ and a fifth vertex adjacent to both midpoints of this $P_4$, as shown in Figure~\ref{fig-bull}. This fifth vertex is called the \emph{nose} of the bull. This graph is also known as $A_0$ in the work of Hertz and de Warra~\cite{hertz:on-the-stabilit:} and as $Q_5$ in the work of Cournier and Ille~\cite{cournier:minimal-indecom:}, and its importance to the modular decomposition is given by the following result.

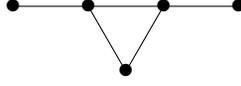
\begin{figure}
	\begin{center}
 	\begin{tikzpicture}[yscale=-1]
		% Edges:
 		\draw (1,0)--(4,0);
 		\draw (2,0)--(2.5,{sqrt(3)/2})--(3,0);
 		% Vertices:
 		\absdot{(1,0)};
 		\absdot{(2,0)};
 		\absdot{(3,0)};
 		\absdot{(4,0)};
 		\absdot{(2.5,{sqrt(3)/2})};
 	\end{tikzpicture}
	\end{center}
	\caption{The bull graph, with its nose on the bottom.}
	\label{fig-bull}
\end{figure}

\begin{theorem}[Cournier and Ille~\cite{cournier:minimal-indecom:}]
\label{thm-path-bull}
Every vertex in a prime graph on four or more vertices lies in an induced $P_4$ or is the nose of an induced bull.
\end{theorem}
	
\section{Obstructions to Bounded Lettericity}
\label{sec-stacked-paths}

The statement of our main result features three families of graphs that serve as obstructions to bounded lettericity: matchings, co-matchings, and stacked paths. The lettericity of the perfect matching $mK_2$ is precisely $m$, as was claimed without proof by Petkov{\v{s}}ek~\cite[Section~5.3]{petkovsek:letter-graphs-a:} and proved by Alecu, Lozin, and de Werra~\cite[Lemma~3]{alecu:the-micro-world:}. Below we include a short sketch showing that this lettericity is simply unbounded, as it parallels the approach we take with stacked paths immediately thereafter.

\begin{proposition}
\label{prop-lettericity-matching}
The set of perfect matchings $\{mK_2\st r\ge 1\}$ has unbounded lettericity.
\end{proposition}
\begin{proof}
Suppose to the contrary that the lettericity of every matching was bounded by some constant, say $r$. In such an $r$-lettering of a matching, every edge has only $r^2$ possible letterings. Therefore the matching $(mr^2+1)K_2$ must contain $m$ edges with the same labels. It follows that the matching $mK_2$ actually has lettericity $2$. However, it is not difficult to verify that $3K_2$ has lettericity $3$, not $2$.
\end{proof}

Because $\ell(\overline{G})=\ell(G)$, it follows that the lettericity of the family of co-matchings is also unbounded.

It remains to define our final family of restrictions, the \emph{stacked paths}, and show that they have unbounded lettericity. These graphs may be defined inductively via the inflation operation as follows. The first stacked path $R_1$ is the path $P_4$. For $n\ge 2$, the stacked path $R_n$ is obtained by inflating the nose of the bull graph by $R_{n-1}$ (and inflating each of the other vertices by a single vertex). Examples are shown in Figure~\ref{fig-stacked-paths}. From the way these examples are drawn, we see that the stacked path $R_n$ can be viewed as consisting of $n$ copies of $P_4$, which we call its \emph{levels}. (Indeed, it can be proved that these are the only induced copies of $P_4$ in $R_n$, but we do not need this fact.)
	
	The stacked path $R_n$ can also be described as the graph on a clique $C$ and a co-clique $S$, where $V(C) = \{c_{1,1}, c_{1,2}, c_{2,1},\dots,c_{n,2}\}$, $V(S) = \{s_{1,1}, s_{1,2}, s_{2,1},\dots,s_{n,2}\}$, and $s_{u_1,v_1} \sim c_{u_2,v_2}$ if and only if $u_1 \geq u_2$ or both $u_1 = u_2$ and $v_1 = v_2$. From this latter definition, it is clear that the stacked paths are split graphs. This is one way to see that they contain neither $2K_2$ nor $\overline{2K_2}$ as induced subgraphs.

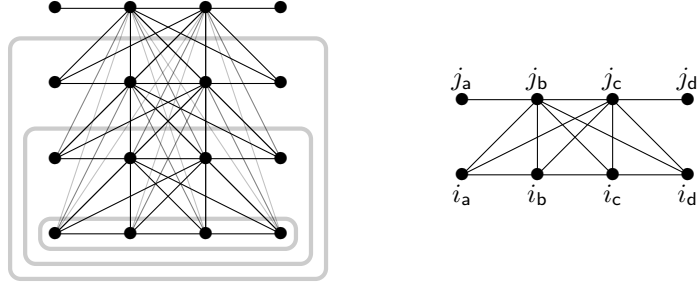
\begin{figure}
	\begin{center}
	 	\begin{tikzpicture}[yscale=-1, baseline=(current bounding box.center)]
	 		% Rectangles (get drawn on top of):
			\draw [color=lightgray, ultra thick, rounded corners]
				(0.4,1.4) rectangle (4.6,4.6)
				(0.6,2.6) rectangle (4.4,4.4)
				(0.8,3.8) rectangle (4.2,4.2);
	 		% First path:
			\foreach \x in {1,2,3,4} {
 				\absdot{(\x,1)};
			}
	 		% First edges:
	 		\draw (1,1)--(4,1);
 			\foreach \x in {2,3} {
				\foreach \i in {1,2,3,4} {
					\foreach \j in {2,3,4} {
						\draw [opacity={pow(1/2,\j-2)}] (\x,1)--(\i,\j);
						%\draw (\x,1)--(\i,\j);
					}
				}
			}
	 		% Second path:
 			\foreach \x in {1,2,3,4} {
 				\absdot{(\x,2)};
			}
	 		% Second edges:
	 		\draw (1,2)--(4,2);
 			\foreach \x in {2,3} {
				\foreach \i in {1,2,3,4} {
					\foreach \j in {3,4} {
						\draw [opacity={pow(1/2,\j-3)}] (\x,2)--(\i,\j);
						%\draw (\x,2)--(\i,\j);
					}
				}
			}
	 		% Third path:
 			\foreach \x in {1,2,3,4} {
 				\absdot{(\x,3)};
			}
	 		% Third edges:
	 		\draw (1,3)--(4,3);
 			\foreach \x in {2,3} {
				\foreach \i in {1,2,3,4} {
					\foreach \j in {4} {
						\draw (\x,3)--(\i,\j);
					}
				}
			}
	 		% Fourth path:
 			\foreach \x in {1,2,3,4} {
 				\absdot{(\x,4)};
			}
	 		% Fourth edges:
	 		\draw (1,4)--(4,4);
	 	\end{tikzpicture}
	 	\quad\quad\quad\quad
		\begin{tikzpicture}[yscale=-1, baseline=(current bounding box.center)]
	 		% First path:
			\foreach \x in {1,2,3,4} {
 				\absdot{(\x,1)};
			}
	 		% First edges:
	 		\draw (1,1)--(4,1);
 			\foreach \x in {2,3} {
				\foreach \i in {1,2,3,4} {
					\foreach \j in {2} {
						\draw (\x,1)--(\i,\j);
					}
				}
			}
	 		% Second path:
 			\foreach \x in {1,2,3,4} {
 				\absdot{(\x,2)};
			}
	 		% Second edges:
	 		\draw (1,2)--(4,2);
	 		% Labels:
	 		\draw (1,1) node[above] {$j_{\mathsf{a}}$};
	 		\draw (2,1) node[above] {$j_{\mathsf{b}}$};
	 		\draw (3,1) node[above] {$j_{\mathsf{c}}$};
	 		\draw (4,1) node[above] {$j_{\mathsf{d}}$};
	 		\draw (1,2) node[below] {$i_{\mathsf{a}}$};
	 		\draw (2,2) node[below] {$i_{\mathsf{b}}$};
	 		\draw (3,2) node[below] {$i_{\mathsf{c}}$};
	 		\draw (4,2) node[below] {$i_{\mathsf{d}}$};
	 	\end{tikzpicture}
	\end{center}
	\caption{On the left, the stacked path $R_4$. On the right, the stacked path $R_2$, labeled as described in the proof of Proposition~\ref{prop-R2-same-letterings}.}
	\label{fig-stacked-paths}
\end{figure}
	
%	\begin{proposition}
%		The stacked path $R_n$ contains precisely $n$ induced copies of $P_4$.
%	\end{proposition}
%	\begin{proof}
%		That $R_n$ contains at least $n$ induced copies of $P_4$ is obvious, as each set of vertices
%		\[
%			\{s_{i,1}, c_{i,1},  c_{i,2}, s_{i,2}\}
%		\]
%		forms an induced $P_4$. We refer to these copies $P_4$ as the \emph{levels} of $R_n$, and say that a level is \emph{above} (resp., \emph{below}) another if the first coordinate of its vertices is larger (resp., smaller).
%		
%		Consider an arbitrary induced $P_4$ in $R_n$. Since $C$ forms a clique and $S$ a co-clique, our $P_4$  must consist of two vertices from $C$ and two from $S$, with the midpoints of our $P_4$ lying in $C$, and the endpoints in $S$. But then each endpoint must lie in a level above (or at the same level as) its neighboring midpoint, but below (or at the same level as) its nonneighboring midpoint. It then must be that the four vertices lie in the same level, and thus that our induced $P_4$ is a level of our stacked path. Since $R_n$ has exactly $n$ levels, our proof is complete.
%	\end{proof}

We show that the family of stacked paths has unbounded lettericity with our next two results.

\begin{proposition}
\label{prop-stacked-path-same-letterings}
	If the family of stacked paths were to have bounded lettericity, then it would have lettericity at most four. Moreover, for every positive integer $n$, there would be a lettering of $R_n$ in which all of the vertices $\{s_{i,1}\st 1\le i\le n\}$ were encoded by the same letter, all of the vertices $\{c_{i,1}\st 1\le i\le n\}$ were encoded by the same letter, all of the vertices $\{c_{i,2}\st 1\le i\le n\}$ were encoded by the same letter, and all of the vertices $\{s_{i,2}\st 1\le i\le n\}$ were encoded by the same letter.
\end{proposition}
\begin{proof}
	Suppose that every stacked path had lettericity at most $k$, for some $k$. In any $k$-lettering of a stacked path, there are $k^4$ possible letterings of each level. Therefore in any $k$-lettering of $R_{(n-1)k^4+1}$, there must be $n$ levels that all use the same letters to encode the same vertices. These levels form an induced copy of $R_n$, and their letterings satisfy the conditions of the proposition.
\end{proof}

We now show that the family of stacked paths cannot have bounded lettericity by showing that $R_2$ does not have a $4$-lettering that satisfies the conditions of Proposition~\ref{prop-stacked-path-same-letterings}. Alternatively, one could show that $R_3$ does not have a $4$-lettering at all.

\begin{proposition}
\label{prop-R2-same-letterings}
	The stacked path $R_2$ has no 4-lettering in which $\{s_{1,1},s_{2,1}\}$ are encoded by the same letter, $\{c_{1,1},c_{2,1}\}$ are encoded by the same letter, $\{c_{1,2},c_{2,2}\}$ are encoded by the same letter, and $\{s_{1,2},s_{2,2}\}$ are encoded by the same letter.
\end{proposition}
\newcommand{\ia}{i_{\mathsf{a}}}
\newcommand{\ib}{i_{\mathsf{b}}}
\newcommand{\ic}{i_{\mathsf{c}}}
\newcommand{\id}{i_{\mathsf{d}}}
\newcommand{\ja}{j_{\mathsf{a}}}
\newcommand{\jb}{j_{\mathsf{b}}}
\newcommand{\jc}{j_{\mathsf{c}}}
\newcommand{\jd}{j_{\mathsf{d}}}
\begin{proof}
	Suppose to the contrary that $R_2$ did have such a lettering. Without loss of generality, we may assume that each level uses four distinct letters, say $\{\mathsf{a},\mathsf{b},\mathsf{c},\mathsf{d}\}$. We further name the vertices by $\ia$, $\ib$, $\dots$, $\jd$ as shown on the right of Figure~\ref{fig-stacked-paths}. Note that by our conventions these are both names of vertices in $\Gamma_D(w)$ and indices of letters in $w$, and thus we argue about their relative values, meaning their positions in $w$.

	By symmetry, we may assume that $\ia<\ja$, as otherwise we may consider the reverse of the word $w$ together with the decoder obtained by reversing all pairs in $D$. Since $\ib$ distinguishes $\ia$ from $\ja$, we must have
	\[
		\ia<\ib<\ja.
	\]
	This implies that $(\mathsf{a},\mathsf{b})\in D$ and that $(\mathsf{b},\mathsf{a})\notin D$. Therefore, since $\jb$ is adjacent to both $\ia$ and $\ja$, we must have
	\[
		\ia<\ib<\ja<\jb.
	\]
	Moving on, we see that $\jc$ distinguishes $\ia$ from $\ja$, so we must have
	\[
		\ia<\ib, \jc<\ja<\jb.
	\]
	We do not know the relative position of $\ib$ and $\jc$, but regardless, we must have $(\mathsf{a},\mathsf{c})\in D$ and $(\mathsf{c},\mathsf{a})\notin D$. As $\ic$ is not adjacent to $\ia$ (or $\ja$), this implies that we must have
	\[
		\ic<\ia<\ib, \jc<\ja<\jb.
	\]		
	Next we see that $\id$ distinguishes $\ib$ from $\jb$, so it must lie between these two. Since this forces $\id$ after $\ic$, and $\id$ is adjacent to both $\ic$ and $\jc$, it must also lie after $\jc$, and thus we have
	\[
		\ic<\ia<\ib, \jc<\ja, \id<\jb.
	\]
	Finally we try to place $\jd$. Since $\jd$ distinguishes $\ic$ from $\jc$, it must lie between them:
	\[
		\ic
		\underbrace{<\ia<\ib, }_{\text{$\jd$ here}}
		\jc<\ja, \id<\jb.
	\]
	However, $\id$ is adjacent to $\jb$, while $\jd$ is not, and we have shown that both must lie before $\jb$, so we have reached a contradiction.
\end{proof}
 
\section{Main Theorem}
\label{sec-main-theorem}

We may now state and prove our main result.

\begin{theorem}
\label{thm-prime-graphs-letters}
Let $\C$ be a class of graphs whose prime members have finite lettericity. If $\C$ contains all matchings, all co-matchings, or all stacked paths, then the lettericity of $\C$ is infinite. Otherwise, the lettericity of $\C$ is finite.
\end{theorem}
\begin{proof}
We have already shown that the lettericity of matchings, co-matchings, and stacked paths is unbounded, and thus so must be the lettericity of any class containing one of these families. It remains to prove the other direction.

For the purposes of this proof, we define a \emph{$(p,q,r)$ graph} to be one that contains none of the graphs $pK_2$, $\overline{qK_2}$, or $R_r$ as induced subgraphs. Let $m$ be a fixed integer. We prove, by induction on $p+q+r$, that there is a function $f(p,q,r)$ so that if every prime induced subgraph of a $(p,q,r)$ graph $G$ has lettericity at most $m$, then the lettericity of $G$ is at most $f(p,q,r)$. We note that no effort has been made to optimize $f(p,q,r)$.

The claim is trivial for $(1,1,r)$ graphs for any $r\ge 1$, as such graphs may contain neither an edge $K_2$ nor a non-edge $\overline{K_2}$ and thus may only be $K_1$, which has lettericity $1$. Thus we may suppose that $p,q\ge 2$, that $p+q+r\ge 5$, and that the claimed function exists for all smaller values of $p+q+r$.

Let $G$ be a $(p,q,r)$ for which the lettericity of all of its prime induced subgraphs is at most $m$. We must first dispense with isolated and dominating vertices. We successively remove such vertices from the graph $G$ until we are left with an induced subgraph $G'$ that has neither type of vertex. It then follows that
\[
	\ell(G)\le\ell(G')+2,
\]
because from any $k$-lettering of $G'$ we can construct a $(k+2)$-lettering of $G$ by adding two new letters, say $\{\mathsf{i},\mathsf{d}\}$, such that $\mathsf{i}$ encodes the isolated vertices and $\mathsf{d}$ encodes the dominating vertices, as in the $2$-letterings of threshold graphs given in Section~\ref{sec-lettericity}.

We now apply the modular decomposition (Theorem~\ref{thm-mod-decomp}) to this graph $G'$ to see that
\[
	G' = H[G_v\st v \in H],
\]
where $H$ is a prime induced subgraph of $G$. Our hypotheses imply that the lettericity of $H$ is at most $m$, so there is some alphabet $\Sigma_H$ of cardinality at most $m$, a decoder $D_H\subseteq\Sigma_H^2$, and a word $w_H$ of length $|V(H)|$ such that $\Gamma_{D_H}(w_H)$ is isomorphic to $H$. We will produce an encoding of $G'$ over a larger alphabet with a related decoder in which each letter of $w_H$ is expanded by a word that encodes the corresponding module of $G'$.

We first examine the case where $H$ has precisely two vertices, so $G' = G_1 \uplus G_2$. Because $G'$ does not have isolated vertices, each of $G_1$ and $G_2$ must contain an edge. This implies that both of these graphs must be $(p-1,q,r)$ graphs, as otherwise $G'$ (and hence also $G$) would contain an induced $pK_2$ subgraph. Thus we can encode each of $G_1$ and $G_2$ with their own alphabet of cardinality at most $f(p-1,q,r)$, and this shows that
\[
	\ell(G)
	\le
	\ell(G')+2
	\le
	\ell(G_1) + \ell(G_2) + 2
	\le
	2f(p-1,q,r)+2.
\]
In the symmetric case where $G'=\overline{G_1 \uplus G_2}$, we obtain the similar bound
\[
	\ell(G)
	\le
	\ell(G')+2
	\le
	\ell(G_1) + \ell(G_2) + 2
	\le
	2f(p,q-1,r)+2.
\]

\begin{figure}
	\begin{center}
	\begin{tikzpicture}
 		% Diamond around vertex:
 		\draw (1,0)
 			+(0.25,0)-- +(0,0.25)-- +(-0.25,0)-- +(0,-0.25) --cycle;
 		% Ellipse around two vertices:
		\draw[color=lightgray, fill=lightgray] (3,0.25)
			arc (90:270:0.25) % left arc
			--(4,-0.25)
			arc (-90:90:0.25) % right arc
			--cycle;
		% Edges:
 		\draw (1,0)--(4,0);
 		% Vertices:
 		\absdot{(1,0)};
 		\absdot{(2,0)};
 		\absdot{(3,0)};
 		\absdot{(4,0)};
 	\end{tikzpicture}
 	\quad\quad\quad\quad
 	\begin{tikzpicture}
 		% Diamond around vertex:
 		\draw (2,0)
 			+(0.25,0)-- +(0,0.25)-- +(-0.25,0)-- +(0,-0.25) --cycle;
 		% Shape around two vertices:
		\draw[color=lightgray, fill=lightgray] (0.75,0)
			arc (-180:0:0.25) % left arc
			arc (180:0:0.75)
			arc (-180:0:0.25)
			arc (0:180:1.25);
		% Edges:
 		\draw (1,0)--(4,0);
 		% Vertices:
 		\absdot{(1,0)};
 		\absdot{(2,0)};
 		\absdot{(3,0)};
 		\absdot{(4,0)};
 	\end{tikzpicture}
 	\quad\quad\quad\quad
 	\begin{tikzpicture}[yscale=-1]
 		% Diamond around vertex:
 		\draw (2.5,{sqrt(3)/2})
 			+(0.25,0)-- +(0,0.25)-- +(-0.25,0)-- +(0,-0.25) --cycle;
 		% Ellipse around four vertices:
		\draw[color=lightgray, fill=lightgray] (1,0.25)
			arc (90:270:0.25) % left arc
			--(4,-0.25)
			arc (-90:90:0.25) % right arc
			--cycle;
		% Edges:
 		\draw (1,0)--(4,0);
 		\draw (2,0)--(2.5,{sqrt(3)/2})--(3,0);
 		% Vertices:
 		\absdot{(1,0)};
 		\absdot{(2,0)};
 		\absdot{(3,0)};
 		\absdot{(4,0)};
 		\absdot{(2.5,{sqrt(3)/2})};
 	\end{tikzpicture}
 	\end{center}
 	\caption{The three cases in the proof of Theorem~\ref{thm-prime-graphs-letters}; the vertex of interest is enclosed in a diamond, while the shading indicates the relevant edge, non-edge, or path of length four.}
 	\label{fig-path-bull-cases}
 \end{figure}
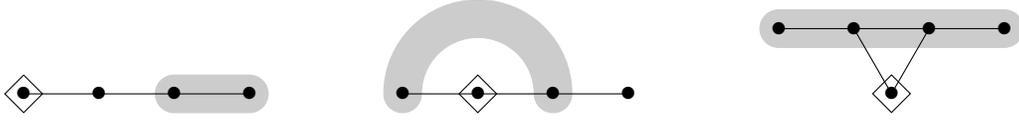

Now suppose that $H$ has more than two vertices (and so $|V(H)| \geq 4$, as there are no prime graphs on three vertices). By Theorem~\ref{thm-path-bull}, every vertex $v\in V(H)$ is either the endpoint of an induced $P_4$, the midpoint of an induced $P_4$, or the nose of a bull. In each of these three cases (which are shown in Figure~\ref{fig-path-bull-cases}), we have a may appeal to induction to bound the lettericity of $G_v$:
\begin{itemize}[topsep=-4pt, itemsep=-2pt]
\item if $v$ is the endpoint of a $P_4$, then $G_v$ must be a $(p-1,q,r)$ graph, as otherwise $G_v$ together with any two vertices of $G'$ that correspond to the non-neighbors of $v$ in this copy of $P_4$ in $H$ would create a copy of $pK_2$ in $G'$;
\item if $v$ is the midpoint of a $P_4$, then $G_v$ must be a $(p,q-1,r)$ graph, as otherwise $G_v$ together with any two vertices of $G'$ that correspond to the neighbors of $v$ in this copy of $P_4$ in $H$ would create a copy of $\overline{qK_2}$ in $G'$;
\item if $v$ is the nose of a bull, then $G_v$ must be a $(p,q,r-1)$ graph, as otherwise $G_v$ together with any four vertices of $G'$ that correspond to the $P_4$ in this copy of the bull in $H$ would create a copy of $R_r$ in $G$.
\end{itemize}

This bounds the lettericity of every module of $G'$, but to bound the lettericity of $G'$ itself we must consider the homogeneous and non-homogeneous modules separately. Define the set $A\subseteq V(H)$ by
\[
	A=\{v\in V(H)\st\text{$G_v$ is neither complete nor edgeless}\},
\]
so it consists of the vertices of $H$ that are inflated by non-homogeneous modules. As each of these modules contains both an edge $K_2$ and a non-edge $\overline{K_2}$, the induced subgraph $H[A]$ cannot contain $K_q$ or $\overline{K_p}$ (in the first case, $G'$ would contain $\overline{qK_2}$, while in the second case, $G'$ would contain $pK_2$). Therefore the cardinality of $A$ is less than the Ramsey number $R(p,q)$. As we have already established that each of the modules $G_v$ for $v\in A$ has bounded lettericity, and we have just established that $A$ has bounded size, we may use disjoint sets of letters to encode each of these modules. Doing so requires at most
\[
	g(p,q,r)
	=
	(R(p,q)-1)
	\cdot
	\max\{ f(p-1,q,r), f(p,q-1,r), f(p,q,r-1) \}
\]
letters.

It remains only to encode the homogeneous modules of $G'$---those that correspond to vertices of $B=V(H)\setminus A$. Choose an $m$-lettering of $H$ (such a lettering exists by our hypotheses). For each letter $a$ in this lettering, there are two cases. If $(a,a)$ lies in the decoder (for $H$), then
\begin{itemize}[topsep=-4pt, itemsep=-2pt]
\item every vertex of every complete module $G_v$ where $v\in B$ is encoded by the letter $a$ may be encoded by the same letter;
\item there are at most $q-1$ edgeless modules $G_v$ where $v\in B$ is encoded by the letter $a$ (as otherwise $G'$ would contain $\overline{qK_2}$), and each of them may be encoded by its own letter.
\end{itemize}
Thus all of the vertices of all of these modules require $q$ letters in total. If $(a,a)$ does not lie in the decoder, then, symmetrically,
\begin{itemize}[topsep=-4pt, itemsep=-2pt]
\item every vertex of every edgeless module $G_v$ where $v\in B$ is encoded by the letter $a$ may be encoded by the same letter;
\item there are at most $p-1$ complete modules $G_v$ where $v\in B$ is encoded by the letter $a$ (as otherwise $G'$ would contain $pK_2$), and each of them may be encoded by its own letter.
\end{itemize}
Thus all of the vertices of all of these modules require $p$ letters in total.

Combining the homogeneous and non-homogeneous modules, we obtain the bound
\[
	\ell(G)
	\le
	\ell(G')+2
	\le
	g(p,q,r)+p+q+2.
\]
This shows that we may take $f(p,q,r)=g(p,q,r)+p+q+2$, and thus completes the proof of the theorem.
\end{proof}

\section{Concluding Remarks}

As the original motivation for the introduction of letter graphs was the study of well-quasi-order, it might be hoped that Theorem~\ref{thm-prime-graphs-letters} would have an application to well-quasi-order. However, its possible implication in that area already follows by more general machinery. Atminas and Lozin~\cite[Theorem~4]{atminas:labelled-induce:} prove that every set of graphs of bounded lettericity is in fact labeled well-quasi-ordered (also known as being well-quasi-ordered by the labeled induced subgraph relation), and they also prove \cite[Theorem~2]{atminas:labelled-induce:} that if the prime graphs in a graph class are labeled well-quasi-ordered, then the entire graph class is labeled well-quasi-ordered. This is stronger than the well-quasi-order implication of our Theorem~\ref{thm-prime-graphs-letters}. While Theorem~\ref{thm-prime-graphs-letters} doesn't have a useful implication for the study of well-quasi-order, we are hopeful that it may be applicable to the characterization of graph classes of bounded lettericity.

Letter graphs are known, via the results of \cite{alecu:letter-graphs-a:abstract,alecu:letter-graphs-a:,alecu:letter-graphs-a:iff}, to have a strong connection to the study of permutation patterns, and here Theorem~\ref{thm-prime-graphs-letters} does have a novel corollary (though again in this context, its well-quasi-order implications are already known, via more general results of Brignall and Vatter~\cite{brignall:labeled-well-qu:}).

In order to keep our account of this application brief, we state it without properly defining the terms, which may be found in the survey~\cite{vatter:permutation-cla:}. Alecu, Ferguson, Kant\'e, Lozin, Vatter, and Zamaraev~\cite{alecu:letter-graphs-a:iff} prove that a permutation class is geometrically griddable if and only if the corresponding graph class has bounded lettericity. Translating our Theorem~\ref{thm-prime-graphs-letters} and combining it with this result yields the following; see Figure~\ref{fig-nested-2413} for a picture of the nestings of $2413$ and $3142$ that appear in the statement of the result.

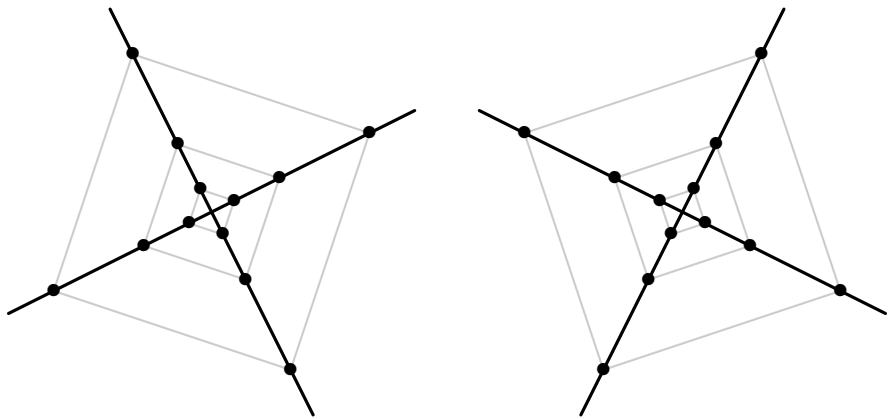
\begin{figure}
	\begin{center}
	\begin{tikzpicture}[scale=0.15, baseline=(current bounding box.center)]
		% "Grid" lines:
		\draw [very thick, line cap=round] (0,9)--(36,27);
		\draw [very thick, line cap=round] (9,36)--(27,0);
		% Inner copy:
		\draw [thick, lightgray] (16,17)--(17,20)--(20,19)--(19,16)--cycle;
 		\plotpartialperm{16/17, 17/20, 20/19, 19/16};
		% Middle copy:
 		\draw [thick, lightgray] (12,15)--(15,24)--(24,21)--(21,12)--cycle;
		\plotpartialperm{12/15, 15/24, 24/21, 21/12};
		% Outer copy:
 		\draw [thick, lightgray] (4,11)--(11,32)--(32,25)--(25,4)--cycle;
 		\plotpartialperm{4/11, 11/32, 32/25, 25/4};
 	\end{tikzpicture}
 	\quad\quad
	\begin{tikzpicture}[scale=0.15, xscale=-1, baseline=(current bounding box.center)]
		% "Grid" lines:
		\draw [very thick, line cap=round] (0,9)--(36,27);
		\draw [very thick, line cap=round] (9,36)--(27,0);
		% Inner copy:
		\draw [thick, lightgray] (16,17)--(17,20)--(20,19)--(19,16)--cycle;
 		\plotpartialperm{16/17, 17/20, 20/19, 19/16};
		% Middle copy:
 		\draw [thick, lightgray] (12,15)--(15,24)--(24,21)--(21,12)--cycle;
		\plotpartialperm{12/15, 15/24, 24/21, 21/12};
		% Outer copy:
 		\draw [thick, lightgray] (4,11)--(11,32)--(32,25)--(25,4)--cycle;
 		\plotpartialperm{4/11, 11/32, 32/25, 25/4};
 	\end{tikzpicture}
	\end{center}
\caption{Renditions of nestings of $2413$ (left) and of $3142$ (right).}
\label{fig-nested-2413}
\end{figure}

\begin{corollary}
Let $\C$ be a class of permutations and suppose that the simple permutations of $\C$ lie in a common geometric grid class. If $\C$ contains arbitrarily long sums of $21$, arbitrarily long skew sums of $12$, or arbitrarily long nestings of $2413$ or $3142$, then $\C$ is not geometrically griddable. Otherwise, $\C$ is geometrically griddable.
\end{corollary}

\section*{Acknowledgements}
We are grateful to the anonymous referees for their comments that greatly improved the paper. We particularly appreciated the first referee's suggestion that we expand the scope of the main result from classes with only finitely many prime graphs to classes whose prime graphs have bounded lettericity, and his thoughts on how to extend our proof to that context.

%
%
%
%
%
%
%

%\bibliographystyle{acm}
%\bibliography{../../refs}

\def\cprime{$'$}

\end{document}